\documentclass[12pt,oneside,reqno]{amsart}
\usepackage {amssymb}
\usepackage {amsmath}
\usepackage {amsmath}
\usepackage {amsthm}
\usepackage{graphicx}
\usepackage{multirow}
\usepackage{xypic}
\usepackage {amscd}
\usepackage{mathrsfs}
\usepackage[colorlinks, linkcolor=blue, anchorcolor=black, citecolor=red]{hyperref}

\usepackage{geometry}  

\newcommand{\Fut}{\textrm{Fut}}

\newcommand{\mcL}{{\mathcal{L}}}
\newcommand{\mcX}{{\mathcal{X}}}

\newcommand{\vol}{\mathrm{vol}}
\newcommand{\ord}{\textrm{ord}}
\newcommand{\lct}{\textrm{lct}}
\newcommand{\vphi}{\varphi}

\newcommand{\bC}{{\mathbb{C}}}

\newcommand{\FS}{\mathrm{FS}}
\newcommand{\bP}{{\mathbb{P}}}

\newcommand{\sslash}{{/\!/}}

\newcommand{\bT}{\mathbb{T}}

\newcommand{\NA}{\mathrm{NA}}
\newcommand{\MA}{\mathrm{MA}}
\newcommand{\mcJ}{{\mathcal{J}}}

\newcommand{\bQ}{{\mathbb{Q}}}

\newcommand{\bR}{{\mathbb{R}}}

\newcommand{\mcH}{{\mathcal{H}}}
\newcommand{\mcO}{{\mathcal{O}}}
\newcommand{\mcF}{{\mathcal{F}}}
\newcommand{\bZ}{{\mathbb{Z}}}

\newcommand{\bN}{{\mathbb{N}}}

\newcommand{\la}{\langle}
\newcommand{\ra}{\rangle}

\newcommand{\bfL}{\mathbf{L}}

\newcommand{\Val}{\textrm{Val}}

\newcommand{\bfD}{\mathbf{D}}

\newcommand{\bfM}{\mathbf{M}}
\newcommand{\bfH}{\mathbf{H}}
\newcommand{\bfE}{\mathbf{E}}
\newcommand{\bfJ}{\mathbf{J}}

\newcommand{\bfF}{\mathbf{F}}

\newcommand{\Aut}{\textrm{Aut}}
\newcommand{\PSH}{\textrm{PSH}}

\newcommand{\Lam}{{\mathbf{\Lambda}}}
\newcommand{\cF}{\mathcal{F}}
\newcommand{\bfI}{\mathbf{I}}

\newcommand{\ud}{\underline}

\newcommand{\fa}{\mathfrak{a}}
\newcommand{\cE}{\mathcal{E}}

\newcommand{\ddc}{\mathrm{dd^c}}

\newcommand{\bfm}{\mathbf{m}}

\newcommand{\bV}{\mathbf{V}}

\newcommand{\hvol}{\widehat{\mathrm{vol}}}
\newcommand{\ocX}{\overline{\mcX}}
\newcommand{\ocL}{\overline{\mcL}}

\theoremstyle{plain}
\newtheorem{theorem}{Theorem}[section]

\theoremstyle{definition}
\newtheorem{defn}[theorem]{Definition}
\newtheorem{exmp}[theorem]{Example}

\newtheorem{conj}[theorem]{Conjecture}

\numberwithin{equation}{section}

\begin{document}


\title{Canonical K\"{a}hler metrics and stability of algebraic varieties}
\author{Chi Li}



\begin{abstract}

We survey some recent developments in the study of canonical K\"{a}hler metrics on algebraic varieties and their relation with stability in algebraic geometry. 

\end{abstract}

\maketitle

\setcounter{tocdepth}{1}
\tableofcontents


The study of canonical K\"{a}hler metrics on algebraic varieties is a very active program in complex geometry. It is a common playground of several fields: differential geometry, partial differential equations, pluripotential theory, birational algebraic geometry and non-Archimedean analysis. We will try to give the reader a tour of this vast program, emphasizing recent developments and highlighting interactions of different concepts and techniques. This article consists of three parts. In the \textit{first} part, we discuss important classes of canonical K\"{a}hler metrics, and explain a well-established variational formalism for studying their existence.  
In the \textit{second} part, we discuss algebraic aspects by reviewing recent developments in the study of K-stability with the help of deep tools from algebraic geometry and non-Archimedean analysis. 
In the \textit{third} part, we discuss how the previous two parts are connected with each other. In particular we will discuss the Yau-Tian-Donaldson (YTD) conjecture for canonical K\"{a}hler metrics in the first part.

\section{Canonical K\"{a}hler metrics on algebraic varieties}
\subsection{Constant scalar curvature K\"{a}hler metrics}
Let $X$ be an $n$-dimensional projective manifold equipped with an ample line bundle $L$. By Kodaira's theorem, we have an embedding $\iota_m: X\rightarrow \bP^{N}$ by using a complete linear system $|mL|$ for $m\gg 1$. If we denote by $h_\FS$ the standard Fubini-Study metric on the hyperplane bundle over $\bP^N$ with Chern curvature $\omega_\FS=-\ddc \log h_\FS$, then $h_0=\iota_m^*h_\FS^{1/m}$ is a smooth Hermitian metric on $L$ whose Chern curvature $\omega_0=\frac{1}{m}\iota_m^*\omega_\FS=-\ddc \log h_0$ is a K\"{a}hler form in $c_1(L)\in H^2(X,\bR)$. In this paper we will use the convention $\ddc=\frac{\sqrt{-1}}{2\pi}\partial\bar{\partial}$. 

We will also use singular Hermitian metrics. 
An upper semicontinuous function $\vphi\in L^{1}(\omega^n)$ is called a $\omega_0$-psh potential if $\psi+\vphi$ is a plurisubharmonic function for any local potential $\psi$ of $\omega_0$ (i.e. $\omega_0=\ddc\psi$ locally). $h_\vphi:=h_0 e^{-\vphi}$ is then called a psh Hermitian metric on $L$. Denote by $\PSH(\omega_0)$ the space of $\omega_0$-psh functions. 
By a $\partial\bar{\partial}$-lemma, any closed positive $(1,1)$-current in $c_1(L)$ is of the form $\omega_\vphi:=\omega_0+\ddc \vphi=-\ddc \log h_\vphi$ with $\vphi\in \PSH(\omega_0)$. Moreover $\omega_{\vphi_2}=\omega_{\vphi_1}$ if and only if $\vphi_2-\vphi_1$ is a constant. 
Define the space of smooth strictly $\omega_0$-psh potentials (also called K\"{a}hler potentials):
\begin{equation}\label{eq-cH}
\mcH:=\mcH(\omega_0)=\{\vphi\in C^\infty(X); \omega_\vphi=\omega_0+\ddc\vphi>0\}.
\end{equation}
Fix any $\vphi\in \mcH$, if $\omega_\vphi=\sqrt{-1}\sum_{i,j}(\omega_\vphi)_{i\bar{j}}dz_i\wedge d\bar{z}_j$ under a holomorphic coordinate chart, its \textit{Ricci curvature} form $Ric(\omega_\vphi)=\frac{\sqrt{-1}}{2\pi}\sum_{i,j} R_{i\bar{j}}dz_i\wedge d\bar{z}_j$ has the coefficients given by:
\begin{equation*}
R_{i\bar{j}}:=Ric(\omega_\vphi)_{i\bar{j}}=-\frac{\partial^2\log \det((\omega_\vphi)_{k\bar{l}}) }{\partial z_i\partial \bar{z}_j}.
\end{equation*}
$Ric(\omega_\vphi)$ is a real closed $(1,1)$-form which represents the cohomology class $c_1(-K_X)=:c_1(X)$. Here $-K_X=\wedge^n T^{(1,0)}X$ is the anti-canonical line bundle of $X$. The \textit{scalar curvature} of $\omega_\vphi$ is given by the contraction:
\begin{equation*}\label{eq-Scalar}
S(\omega_\vphi)=\omega_\vphi^{i\bar{j}}(Ric(\omega_\vphi))_{i\bar{j}}=\frac{n\cdot Ric(\omega_\vphi)\wedge \omega_\vphi^{n-1}}{\omega_\vphi^n}.
\end{equation*}
$\omega_\vphi$ is called a \textit{constant scalar curvature K\"{a}hler} (cscK) metric if $S(\omega_\vphi)$ is the constant $\ud{S}$ which is the average scalar curvature and is determined by cohomology calsses:
\begin{equation}\label{eq-udSV}
\ud{S}=\frac{n \la c_1(X)\cdot c_1(L)^{\cdot n-1}, [X]\ra}{\bV} \quad \text{ with } \quad \bV=\la c_1(L)^{\cdot n}, [X]\ra. 
\end{equation}
The K\"{a}hler potential of a cscK metric is a solution to a 4-th order nonlinear PDE. In general, there are obstructions to the existence of cscK metrics. For example, the Matsushima-Lichnerowicz theorem states that if $(X, L)$ admits a cscK metric then the automorphism group $\Aut(X, L)$ must be reductive. Our goal is to discuss the Yau-Tian-Donaldson conjecture which would provide a sufficient and necessary algebraic criterion for the existence of cscK metrics.

\subsection{K\"{a}hler-Einstein metrics and weighted K\"{a}hler-Ricci soliton}
K\"{a}hler-Einstein metrics form an important class of cscK metrics. $\omega_\vphi$ is called K\"{a}hler-Einstein (KE) if $Ric(\omega_\vphi)=\lambda \omega_\vphi$ for a real constant $\lambda$. An immediate necessary condition for the existence of KE metric is that the cohomology class $c_1(X)\in H^2(X, \bR)$ is either negative, numerically trivial or positive. The existence for the first two cases was understood in 70's: there always exists a K\"{a}hler-Einstein metric if $c_1(X)$ is negative (by the work of Aubin and Yau), or if $c_1(X)$ is numerically trivial (by the work of Yau). 

Now we assume that $X$ is a Fano manifold. In other words, $-K_X$ is an ample line bundle and we set $L=-K_X$. 
Any $\vphi\in \mcH$ corresponds to a volume form:  
\begin{equation*}
\Omega_\vphi:=|s^*|^2_{h_\vphi} (\sqrt{-1})^{n^2} s\wedge \bar{s}=\Omega_0 e^{-\vphi}\;\; \text{with} \;\; s=dz_1\wedge \cdots \wedge dz_n, \; s^*=\partial_{z_1}\wedge \cdots \wedge \partial_{z_n}.
 \end{equation*}
The KE equation in this case is reduced to a complex Monge-Amp\`{e}re equation for $\vphi$:
\begin{equation*}
(\omega+\ddc\vphi)^n=e^{-\vphi}\Omega_0.
\end{equation*}
We also consider an interesting generalization of K\"{a}hler-Einstein metrics on Fano manifolds with torus actions.  
Assume that $\bT\cong (\bC^*)^r$ is an algebraic torus and $T\cong (S^1)^r\subset \bT$ is a compact real subtorus. We will use the following notation:
\begin{equation}\label{eq-Nlattice}
N_\bZ=\mathrm{Hom}_{\mathrm{alg}}(\bC^*, \bT), \quad N_\bQ=N_\bZ\otimes_\bZ\bQ, \quad N_\bR=N_\bZ\otimes_\bZ \bR. 
\end{equation}
Assume that $\bT$ acts faithfully on $X$. Then there is an induced $\bT$-action on $-K_X$.
Each $\xi\in N_\bR$ corresponds to a holomorphic vector field $V_\xi$. Assume that $\bT$ acts faithfully on $X$. Then there is an induced $\bT$-action on $-K_X$. Denote by $\mcH^T$ the set of $T$-invariant K\"{a}hler potentials. For any $\vphi\in \mcH^T$, 
the $T$-action becomes Hamiltonian with respect to $\omega_\vphi$. Denote by $\bfm_\vphi: X\rightarrow N_\bR^*\cong \bR^r$ the corresponding moment map, and let $P$ be the image of $\bfm_\vphi$. By a theorem of Atiyah-Guillemin-Sternberg, $P$ is a convex polytope which depends only on the K\"{a}hler class $c_1(L)$. 
Let $g: P\rightarrow \bR$ be a smooth \textit{positive} function. The $g$-soliton equation for $\vphi\in \mcH(-K_X)^T$ is:
\begin{equation*}\label{eq-gsoliton}
g(\bfm_\vphi) (\omega_0+\ddc\vphi)^n= e^{-\vphi}\Omega_0.
\end{equation*}
The equivalent tensorial equation is given by
$
Ric(\omega_\vphi)=\omega_\vphi+\ddc \log g(\bfm_\vphi). 
$
\begin{exmp}\label{exmp-KR}
If $g(y)=e^{-\la y, \xi\ra}$, then the above equation becomes the standard 
K\"{a}hler-Ricci soliton equation $Ric(\omega_\vphi)=\omega_\vphi+\mathscr{L}_{V_\xi}\omega_\vphi$ where $\mathscr{L}$ denotes the Lie derivative.
\end{exmp}

\subsection{K\"{a}hler-Einstein metrics on log Fano pairs}\label{sec-logFano}

Singular algebraic varieties and log pairs are important objects in algebraic geometry, and
appear naturally for studying limits of smooth varieties.   
It is thus natural to study canonical K\"{a}hler metric on general log pairs. 
We recall a definition from birational algebraic geometry. 
Let $X$ be a normal projective variety and $D$ be an $\bQ$-Weil divisor. Assume that $K_X+D$ is $\bQ$-Cartier. Let $\mu: Y\rightarrow X$ be a resolution of singularities of $(X, D)$ with simple normal crossing exceptional divisors $\sum_i E_i$. 
We then have an identity:
\begin{equation}\label{eq-KYKXD}
K_Y=\mu^*(K_X+D)+\sum_i a_i E_i.
\end{equation}
Here $A_{(X,D)}(E_i):=a_i+1$ is called the log discrepancy of $E_i$. The pair $(X, D)$ has klt singularities if $A_{(X,D)}(E_i)>0$ for any $E_i$. 
We will always assume that $(X,D)$ has klt singularities. 

If $K_X+D$ is ample or numerically trivial, Yau and Aubin's existence result had been generalized to the singular and log case 
in \cite{EGZ09}, partly based on Ko\l odziej's pluripotential estimates. There were also many related works by Yau, Tian, H.Tsuji, Z. Zhang and many others. 

Now we assume that $-(K_X+D)$ is ample and call $(X, D)$ a log Fano pair. 
Then one can consider K\"{a}hler-Einstein equation or more generally $g$-soliton equation on $(X, D)$. First note that there is a globally defined volume form as in the smooth case: choose a local trivializing section $s$ of $m(K_X+D)$ with the dual $s^*$ and define $\Omega_0=|s^*|_{h_0}^{2/m} (\sqrt{-1}^{mn^2}s\wedge \bar{s})^{1/m}$. Assume that $\bT$ acts on $X$ faithfully and preserves the divisor $D$. With the notation from before,
we say that $\vphi$ is the potential for a $g$-weighted soliton or just $g$-soliton on $(X, D)$ if $\vphi$ is a bounded $\omega_0$-psh function that satisfies the equation:
\begin{equation}\label{eq-logsoliton}
g(\bfm_\vphi) (\omega+\ddc\vphi)^n=e^{-\vphi}\Omega_0 
\end{equation}
For any bounded $\vphi\in \PSH(\omega_0)$, the $g$-weighted Monge-Amp\`{e}re measure on the left-hand-side of \eqref{eq-logsoliton} is well-defined by the work of Berman-Witt-Nystr\"{o}m \cite{BW14} and also by Han-Li \cite{HL20}, generalizing the definition of Bedford-Taylor (when $g=1$).  
It is known that any bounded solution $\vphi$, if it exists, is orbifold smooth over the orbifold locus of $(X,D)$.  Moreover $p$ is a regular point of $\mathrm{supp}(D)$ such that $D=(1-\beta)\{z_1=0\}$ locally for a holomorphic function $z_1$ (with $\beta\in (0,1]$), then the associated K\"{a}hler metric is modeled by $\bC_\beta\times\bC^{n-1}$ where $\bC_\beta=(\bC, dr^2+\beta^2 r^2 d\theta^2)$ is the 2-dimensional flat cone with cone angle $2\pi \beta$. 

\subsection{Ricci-flat K\"{a}hler cone metrics}\label{sec-RFKC}

The class of Ricci-flat K\"{a}hler cone metrics is closely related to KE/$g$-soliton metrics, and is interesting in both complex geometry and mathematical physics (see \cite{MSY08}). 

Let $Y=\mathbf{Spec}(R)$ be an $(n+1)$-dimensional affine variety with a singularity $o\in Y$. Assume that an algebraic torus $\hat{\bT}\cong (\bC^*)^{r+1}$ acts faithfully on $Y$, with $o$ being the only fixed point. Define $\hat{N}_\bQ, \hat{N}_\bR$ similar to \eqref{eq-Nlattice}. The $\hat{\bT}$-action corresponds to a weight decomposition of the coordinate ring $R=\bigoplus_{\alpha\in \bZ^{r+1}}R_\alpha$.  The Reeb cone can be defined as:
\begin{equation*}
\hat{N}^+_\bR=\big\{\xi\in \hat{N}_\bR; \la \alpha, \xi\ra>0 \text{ for all } \alpha\in \bZ^{r+1}\setminus \{0\} \text{ with } R_\alpha\neq 0 \big\}.
\end{equation*}
Any $\hat{\xi}\in \hat{N}^+_\bR$ is called a Reeb vector and corresponds to an expanding holomorphic vector field $V_{\hat{\xi}}$.  Assume furthermore that $Y$ is $\bQ$-Gorenstein and there is a $\hat{\bT}$-equivariant non-vanishing section $s\in |mK_Y|$, which induces a $\hat{\bT}$-equivariant volume form $dV_Y=(\sqrt{-1}^{m(n+1)^2} s\wedge \bar{s})^{1/m}$ on $Y$.
We call the data $(Y, \hat{\xi})$ with $\hat{\xi}\in \hat{N}^+_\bR$ a polarized Fano cone. 

Let $\hat{T}\cong (S^1)^{r+1}$ be a compact real subtorus of $\hat{\bT}$. 
A $\hat{T}$-invariant function $r: Y\rightarrow \bR_{\ge 0}$ is called a radius function for $\hat{\xi}\in \hat{N}^+_\bR$ if 
$\hat{\omega}=\ddc r^2$ is a K\"{a}hler cone metric on $Y^*=Y\setminus \{o\}$ and $\frac{1}{2}(r\partial_r-\sqrt{-1}J (r\partial_r))=V_{\hat{\xi}}$. Here $J$ is a complex structure on $Y^*$ and $\hat{\omega}$ is called a K\"{a}hler cone metric if $G:=\frac{1}{2}\hat{\omega}(\cdot, J\cdot)$ on $Y^*$ is isometric to $dr^2+r^2 G_S$ where $S=\{r=1\}$ and $G_S=G|_S$. In the literature of CR geometry, the induced structure on $S$ by a K\"{a}hler cone metric is called a Sasaki structure. 
$\hat{\omega}=\ddc r^2$ is called \textit{Ricci-flat} if $Ric(\hat{\omega})=0$. In this case, the radius function satisfies an equation (up to rescaling):
\begin{equation*}\label{eq-RFMA}
(\ddc r^2)^{n+1}=dV_Y. 
\end{equation*}
If $\hat{\xi}\in \hat{N}_\bQ$, then $\hat{\omega}$ is called quasi-regular, and $V_{\hat{\xi}}$ generates a $\bC^*$-subgroup $\la \hat{\xi}\ra$ of $\hat{\bT}$. The GIT quotient $X=Y\sslash \la \hat{\xi}\ra$ admits an orbifold structure encoded by a log Fano pair $(X, D)$. 
A straightforward calculation shows that a quasi-regular $(Y, \hat{\xi})$ admits a Ricci-flat K\"{a}hler cone metric if and only if $(X, D)$ admits a K\"{a}hler-Einstein metric. 

In general there are many irregular Ricci-flat K\"{a}hler cone metrics, i.e. with $\hat{\xi}\in \hat{N}_\bR\setminus \hat{N}_\bQ$. 
Recent works by Apostolov-Calderbank-Jubert-Lahdili establish an equivalence between Ricci-flat K\"{a}hler cone metrics and special $g$-soliton metrics. More precisely, fix any $\hat{\chi}\in \hat{N}^+_\bQ$ and consider the quotient $(X, D)=Y\sslash \la \hat{\chi}\ra$ as above. It is shown in \cite{AJL21} (see also \cite{Li21b}) that the Ricci-flat K\"{a}hler cone metric on $(Y, \hat{\xi})$ is equivalent to the $g$-soliton metric on $(X, D)$ with $g(y)=(n+1+\la y, \xi\ra)^{-n-2}$ where $\xi$ (or equivalently $V_\xi$) is induced by $\hat{\xi}$ on $X$.

\subsection{Analytic criteria for the existence}

We now review a well-understood criterion for the existence of above canonical K\"{a}hler metrics. The general idea is to view corresponding equations as Euler-Lagrange equations of appropriate energy functionals and then use a variational approach to prove that the existence of solutions is equivalent to the coercivity of the energy functionals.  First we have the following functionals defined for any $\vphi\in \mcH$ (see \eqref{eq-cH}).
\begin{align}
&\bfE(\vphi)=\frac{1}{(n+1)\bV}\sum_{k=0}^n \int_X \vphi \omega_\vphi^k\wedge \omega_0^{n-k},\quad \Lam(\vphi)=\frac{1}{\bV}\int_X \vphi \omega_0^n \label{eq-Evphi} \\
&\bfJ(\vphi)=\Lam(\vphi)-\bfE(\vphi), \quad \bfE^\chi(\vphi)=\frac{1}{\bV}\sum_{k=0}^{n-1}\int_X \vphi \chi\wedge \omega_\vphi^k\wedge \omega^{n-1-k}. \label{eq-Jvphi}
\end{align}
Here $\bV$ is defined in \eqref{eq-udSV} and $\chi$ is any closed real $(1,1)$-form. 

The following functionals are important for studying the cscK problem.
\begin{align}
&\bfH(\vphi)=\frac{1}{\bV}\int_X \log \frac{\omega_\vphi^n}{\Omega_0}\omega_\vphi^n,  \quad
\bfM(\vphi)=\bfH(\vphi)+\bfE^{-Ric(\omega_0)}(\vphi)+\ud{S}\cdot \bfE(\vphi).  \label{eq-Mvphi}
\end{align}
$\bfH(\vphi)$ is usually called the entropy of the measure $\omega_\vphi^n$. One can verify that any critical point of $\bfM$ is the potential of a cscK metric. 

For K\"{a}hler-Einstein (KE) metrics on Fano manifolds, we have more functionals: 
\begin{align}\label{eq-Dvphi}
\bfL(\vphi)=-\log \Big(\frac{1}{\bV}\int_X e^{-\vphi}\Omega_0\Big), \quad \bfD(\vphi)=-\bfE(\vphi)+\bfL(\vphi).
\end{align}
A critical point of $\bfD$ is a KE potential. 
These functionals can be generalized to the settings of $g$-weighted solitons and Ricci-flat K\"{a}hler cone metrics (see \cite{Li21b} for references). 

To apply the variational approach, one first needs a `completion' of $\mcH$.
Such completion was defined by Guedj-Zeriahi extending the local study of Cegrell. Following \cite{BBJ18}, one way to introduce this is to first define the $\bfE$ functional for any $\vphi\in \PSH(\omega_0)$:
\begin{equation}\label{eq-Epsh}
\bfE(\vphi)=\inf\{\bfE(\tilde{\vphi}); \tilde{\vphi}\ge \vphi, \quad \tilde{\vphi}\in \mcH(\omega_0)\}, \quad
\end{equation}
Then define the set of finite energy potentials as:
\begin{equation}\label{eq-E1A}
\cE^1:=\cE^1(\omega_0)=\{\vphi\in \PSH(\omega_0); \bfE(\vphi)>-\infty\}.
\end{equation}
After the work \cite{BBEGZ}, $\cE^1$ can be endowed with a strong topology which is the coarsest refinement of the weak topology (i.e. the $L^1$-topology) that makes $\bfE$ continuous. 
The above energy functionals can be extended to $\cE^1$, and they satisfy important regularization properties:
\begin{theorem}[\cite{BBEGZ, BDL17}]\label{thm-Freg}
For any $\vphi\in \cE^1$, there exist $\{\vphi_k\}_{k\in\bN}\subset \mcH$ such that $\bfF(\vphi_k)\rightarrow \bfF(\vphi)$ for $\bfF\in \{\bfE, \Lam, \bfE^{-Ric}, \bfH\}$. 
\end{theorem}
We would like to emphasize the result for $\bfF=\bfH$, which was proved in \cite{BDL17}. The idea of proof there is to first regularize the measure $\omega_\vphi^n$ with converging entropy and then use Yau's solution to complex Monge-Amp\`{e}re equations with prescribed volume forms.  Later we will encounter a same idea later in the non-Archimedean setting. 

Another key concept is the geodesic between two finite energy potentials.
For $\vphi_i, \in \cE^1, i=0,1$,  the \textit{geodesic} connecting them is the following $p_1^*\omega_0$-psh function on $X\times [0, 1]\times S^1$ where $p_1$ is the projection to the first factor (see \cite{BBJ18, Dar15}). 
\begin{equation}\label{eq-geodE1}
\Phi=\sup\big\{\Psi; \; \Psi \text{ is $S^1$-invariant and $p_1^*\omega_0$-psh},\; \lim_{s\rightarrow i}\Psi(\cdot, s)\le \vphi(i), i=0,1 \big\}.
\end{equation}
The concept of geodesic originates from Mabuchi's $L^2$-Riemannian metric on $\mcH$. According to the work of Semmes and Donaldson, if $\vphi_i\in \mcH, i=0,1$, then the geodesic $\Phi$ is a solution to the Dirichlet problem of homogeneous complex Monge-Amp\`{e}re equation:
\begin{equation}\label{eq-geod}
(p_1^*\omega_0+\ddc \Phi)^{n+1}=0, \quad \Phi(\cdot, i)=\vphi_i, i=0,1. 
\end{equation}
Since $\Phi$ is $S^1$-invariant, we can consider $\Phi$ as a family of $\omega_0$-psh functions $\{\vphi(s)\}_{s\in [0,1]}$. 
 \begin{theorem}[\cite{BB17, BDL17}]\label{thm-convex}
 Let $\Phi=\{\vphi(s)\}_{s\in [0,1]}$ be a geodesic segment in $\cE^1$.
$\mathrm{(1)}$ $s\mapsto \bfE(\vphi(s))$ is affine; $\mathrm{(2)}$
$s\mapsto \bfM(\vphi(s))$ is convex. 
\end{theorem}
Theorem \ref{thm-convex} is very important in the variational approach. 
If a geodesic is smooth, the statements follow from straight-forward calculations. However there are examples (first due to Lempert-Vivas) showing that the solution to \eqref{eq-geod} in general does not have sufficient regularity. 
So the proofs of above results are much more involved. 

In this paper $\tilde{\bT}$ will always denote a maximal torus of the linear algebraic group $\Aut(X, L)$ and $\tilde{T}$ is a maximal real subtorus of $\tilde{\bT}$.  In the following result, we use the translation invariance $\bfF(\vphi+c)=\bfF(\vphi)$ for $\bfF\in \{\bfM, \bfJ\}$ and hence $\bfF(\omega_\vphi):=\bfF(\vphi)$ is well-defined.
 \begin{theorem}[\cite{CC21,DR17,BDL20}]\label{thm-cscKcri}
There exists a $\tilde{T}$-invariant cscK metric in $c_1(L)$ if and only if $\bfM$ is reduced coercive, which means that there exist $\gamma, C>0$ such that for any $\vphi\in \mcH^{\tilde{T}}$, 
\begin{equation}\label{eq-redcoer}
\bfM(\omega_\vphi)\ge \gamma \cdot \inf_{\sigma\in \tilde{\bT}} \bfJ(\sigma^*\omega_\vphi)-C. 
\end{equation}
\end{theorem} 
This type of result goes back to Tian's pioneering work in \cite{Tia97} which proves that if $X$ is a Fano manifold with a discrete automorphism group, then the existence of K\"{a}hler-Einstein metric is equivalent to the properness of the $\bfM$-functional, and is also equivalent to the properness of the $\bfD$ functional. Tian's work has since been refined and generalized for other canonical metrics. 
For the necessity direction (from existence to reduced coercivity), there is now a general principle due to Darvas-Rubinstein (\cite{DR17}) that can be applied for all previously mentioned canonical K\"{a}hler metrics. The sufficient direction (from reduced coercivity to existence) for K\"{a}hler-Einstein metrics is re-proved in \cite{BBEGZ} using pluripotential theory, which works equally well in the setting of log Fano pairs. See \cite{BW14, HL20} for the extension to the $g$-soliton case. The existence result for smooth cscK metrics is accomplished recently by Chen-Cheng's new estimates (\cite{CC21}).
The use of maximal torus appears in \cite{Li19, Li20}, refining an earlier formulation of Hisamoto \cite{His16b}. There is also an existence criterion when $\tilde{\bT}$ is replaced by any connected reductive subgroup of $\Aut(X,L)$ that contains a maximal torus.

\section{Stability of algebraic varieties and non-Archimedean geometry}
\subsection{K-stability and non-Archimedean geometry}

The concept of K-stability, as first introduced by Tian and Donaldson, is motivated by results from geometric analysis. On the other hand, the recent development shows that various tools from algebraic geometry are crucial in un-locking many of its mysteries. 
\begin{defn}\label{def-TC}
A test configuration for a polarized manifold $(X, L)$ consists of $(\mcX, \mcL)$ that satisfies:
(i) $\pi: \mcX\rightarrow \bC$ is a flat projective morphism from a normal variety $\mcX$, and $\mcL$ is a $\pi$-semiample $\bQ$-line bundle. 
(ii) There is a $\bC^*$-action on $(\mcX, \mcL)$ such that $\pi$ is $\bC^*$-equivariant. 
(iii)
There is a $\bC^*$-equivariant isomorphism $(\mcX, \mcL)\times_\bC\bC^*\cong (X\times\bC^*, p_1^*L)$. 

Test configuration $(\mcX, \mcL)$ is called a product test configuration if there is a $\bC^*$-equivariant isomorphism $(\mcX, \mcL)\cong (X\times\bC, p_1^*L)$ where the $\bC^*$-action on the right-hand-side is the product action of a $\bC^*$-action on $(X, L)$ with the standard multiplication on $\bC$.

Two test configurations $(\mcX_i, \mcL_i)$, $i=1,2$ are called equivalent if there exists a test configuration $(\mcX', \mcL')$ with $\bC^*$-equivariant birational morphisms $\rho_i: \mcX'\rightarrow \mcX_i$ satisfying $\rho_1^*\mcL_1=\mcL'=\rho_2^*\mcL_2$. 
For any test configuration $(\mcX, \mcL)$, by taking fibre product one can always find an equivalent test configuration $(\mcX', \mcL')$ such that $\mcX'$ dominates $X\times\bC$.

For any test configuration $(\mcX, \mcL)$ there is a canonical compactification over $\bP^1$ denoted by $(\ocX, \ocL)$ which is obtained by adding a trivial fibre over
$\{\infty\}=\bP^1\setminus \bC$. 
\end{defn}

The notion of test configuration is a way to formulate the degeneration of $(X, L)$. In fact any test configuration is induced by a one-parameter subgroup of $\mathrm{PGL}(N+1, \bC)$ for a Kodaira embedding $X\rightarrow \bP^{N}$. 

{We} will continue our discussion in a framework of non-Archimedean geometry as proposed by Boucksom-Jonsson.  
Let $X^\NA$ denote the Berkovich analytification of $X$ {with respect to the trivial absolute value on $\bC$} (see \cite{BFJ15} for references). As a set, $X^\NA$ is a topological space consisting of real valuations on subvarieties of $X$, and contains a dense subset $X^{\mathrm{div}}_\bQ$ consisting of divisorial valuations on $X$. 
Any test configuration $(\mcX, \mcL)$ defines a function on $X^\NA$ in the following way. 
First, up to equivalence we can assume that there is a birational morphism $\rho: \mcX\rightarrow X_\bC:=X\times\bC$. 
Write $\mcL=\rho^*p_1^*L+E$ where $E$ is a $\bQ$-divisor supported on $\mcX_0$. For any $v\in X^\NA$ denote by $G(v)$ the $\bC^*$-invariant semivaluation on $X_\bC$ that satisfies $G(v)|_{\bC(X)}=v$ and $G(v)(t)=1$ where $t$ is the coordinate of $\bC$.   
One then defines:
\begin{equation}\label{eq-phiTC}
\phi_{(\mcX, \mcL)}(v)=G(v)(E), \quad \text{ for any } v\in X^\NA. 
\end{equation}
The set of such functions on $X^\NA$ obtained from test configurations is denoted by $\mcH^\NA$ which is considered as the set of smooth non-Archimedean psh potentials on the analytification of $L$. The following functionals, defined on the space of test configurations, correspond to the Archimedean (i.e. complex analytic) functionals in \eqref{eq-Evphi}-\eqref{eq-Jvphi}. 
\begin{align}
&\bfE^\NA(\mcX, \mcL)=\frac{\ocL^{\cdot n+1}}{(n+1)\bV}, \quad \Lam^\NA(\mcX, \mcL)=\frac{1}{\bV}\ocL^{\cdot n}\cdot \rho^*L_{\bP^1},\label{eq-ENA} \\
&\bfJ^\NA(\mcX, \mcL)=\Lam^\NA(\mcX, \mcL)-\bfE^\NA(\mcX, \mcL), \quad (\bfE^{K_X})^\NA(\mcX, \mcL)=\frac{1}{\bV}K_X\cdot \ocL^{\cdot n}, \label{eq-JNA} \\
& \bfH^\NA=\frac{1}{\bV}K^{\log}_{\ocX/X_{\bP^1}}\cdot \mcL^{\cdot n}, \quad 
\bfM^\NA(\mcX, \mcL)=\bfH^\NA+(\bfE^{K_X})^\NA+\ud{S}\cdot \bfE^\NA  \label{eq-MabNA}
\end{align}
where we assume that $\ocX$ dominates $X_{\bP^1}=X\times\bP^1$ by $\rho$, and $L_{\bP^1}=p_1^*L$, $K^{\log}_{\ocX/X_{\bP^1}}=K_{\ocX}+\mcX^{\mathrm{red}}_0-(\rho^*(K_{X\times\bP^1}+X\times\{0\}))$. 
These functionals were defined before the introduction of the non-Archimedean framework. For example, the $\bfE^\NA$ functional appeared in Mumford's study of Chow stability of projective varieties. 

 Assume that $\mcX_0=\sum_i b_i F_i$ where $F_i$ are irreducible components. Set $v_i=b_i^{-1} \ord_{F_i}\circ p_1^*\in X^{\mathrm{div}}_\bQ$ and let $\delta_{v_i}$ be the Dirac measure supported at $\{v_i\}$. Chambert-Loir defined the following non-Archimedean Monge-Amp\`{e}re measure using the intersection theory.
\begin{equation}\label{eq-TCMA}
\MA^\NA(\phi_{(\mcX, \mcL)})=\sum_i b_i (\mcL^{\cdot n} \cdot F_i) \delta_{v_i}. 
\end{equation}
Mixed non-Archimedean Monge-Amp\`{e}re measures are similarly defined. It then turns out that the functionals from \eqref{eq-ENA}-\eqref{eq-JNA} can be obtained by using the same formula as in \eqref{eq-Evphi}-\eqref{eq-Jvphi} but with the ordinary integrals replaced by corresponding non-Archimedean ones, while the $\bfH^\NA$ has the following expression after \cite{BHJ17}:
\begin{equation}\label{eq-HNA}
\bfH^\NA(\mcX, \mcL)=\int_{X^\NA} A_X(v) \MA^\NA(\phi_{(\mcX, \mcL)})(v).
\end{equation}
Here $A_X$ is a functional defined on $X^{\NA}$ (see \cite{JM12}) that generalizes the log discrepancy functional on $X^{\mathrm{div}}_\bQ$. We can now recall the notion of K-stability:
\begin{defn}\label{def-Kst}
$(X, L)$ is 
$
\Bigg\{
\begin{array}{l}
\text{K-semistable}\\
\text{K-stable}\\
\text{K-polystable}
\end{array}
$
if any non-trivial test configuration $(\mcX, \mcL)$ for $(X, L)$ satisfies:
$
\Bigg\{
\begin{array}{l}
\bfM^\NA(\mcX, \mcL)\ge 0 \\
\bfM^\NA(\mcX, \mcL)>0 \\
\bfM^\NA(\mcX, \mcL)\ge 0 \text{ and } =0 \text{ only if $(\mcX, \mcL)$ is a product test configuration}. 
\end{array}
$
\end{defn}
This is like a Hilbert-Mumford's numerical criterion in the Geometric Invariant Theory. \footnote{In the classical formulation, Tian's CM weight, or equivalently the Donaldson-Futaki invariant is used to define the K-stability. However, to fit our discussion in the non-Archimedean framework, we use the equivalent formulation via the $\bfM^\NA$ functional.} 
The recent development of K-stability involves a strengthened notion called reduced uniform K-stability, 
which matches the reduced coercivity in \eqref{eq-redcoer} (see \cite{BHJ17, Der16, His16b}). Recall that $\tilde{\bT}$ denotes a maximal torus of $\Aut(X, L)$,  and $\tilde{N}_\bQ$ is defined similar to \eqref{eq-Nlattice}. 
\begin{defn}\label{def-uniKst}
$(X, L)$ is 
$
\Bigg\{
\begin{array}{l}
\text{uniformly K-stable}\\
\text{reduced uniformly K-stable}\\
\end{array}
$
if there exists $\gamma>0$ such that any test configuration $(\mcX, \mcL)$ satisfies:
$
\bigg\{
\begin{array}{l}
\bfM^\NA(\mcX, \mcL)\ge \gamma\cdot \bfJ^\NA(\mcX, \mcL) \\
\bfM^\NA(\mcX, \mcL)\ge \gamma\cdot \inf_{\xi\in \tilde{N}_\bQ} \bfJ^\NA(\mcX_\xi, \mcL_\xi). 
\end{array}
$

\end{defn}
Here the twist $(\mcX_\xi, \mcL_\xi)$ is introduced by Hisamoto \cite{His16b}. One way to define it as a test configuration is by resolving the composition of birational morphisms $(\mcX, \mcL)\dasharrow (X_\bC=X\times\bC, L_\bC=p_1^*L)\stackrel{\sigma_{\xi}}{\rightarrow} (X_\bC, L_\bC)$ where $\sigma_{\xi}$ is the $\bC^*$-action generated by $\xi$. Alternatively it can be defined in a more general setting of filtrations (see example \ref{exmp-twist}).

\subsection{Non-Archimedean pluripotential theory}

We discuss how non-Archimedean pluripotential theory as developed by Boucksom-Jonsson can be applied to study K-stability.  
Corresponding to a regularization result in the complex analytic case, an u.s.c. function $\phi: X^\NA\rightarrow \bR\cup \{+\infty\}$ is called a non-Archimedean psh potential if it is a decreasing limit of a sequence from $\mcH^\NA$. Denote the space of such functions by $\PSH^\NA$. 
Boucksom-Jonsson introduced the following non-Archimedean version of the finite energy space. First corresponding to \eqref{eq-Epsh}, for any $\phi\in \PSH^\NA$, define:
\begin{equation*}
\bfE^\NA(\phi)=\inf\{\bfE^\NA(\tilde{\phi}); \tilde{\phi}\ge \phi, \tilde{\phi}\in \mcH^\NA\}. 
\end{equation*}
Then, corresponding to \eqref{eq-E1A}, define the space of non-Archimedean finite energy potentials:
\begin{equation*}\label{eq-E1NA}
(\cE^1)^\NA=\{\phi\in \PSH^\NA; \bfE^\NA(\phi)>-\infty\}. 
\end{equation*}
This space is again equipped with a strong topology which makes $\bfE^\NA$ continuous. 
Boucksom-Jonsson in \cite{BoJ18a} showed that  non-Archimedean Monge-Amp\`{e}re measure $\MA^\NA(\phi)$ is well-defined for any $\phi\in (\cE^1)^\NA$ such that if $\{\phi_k\}_{k\in \bN}\subset \mcH^\NA$ converges to $\phi$ strongly, then $\MA^\NA(\phi_k)$ converges to $\MA^\NA(\phi)$ weakly. 

A large class of potentials come from filtrations (see \cite{BHJ17}). Set $R_m=H^0(X, mL)$.  
\begin{defn}
A filtration is the data $\cF=\{\cF^\lambda R_m \subseteq R_m; \lambda\in \bR, m\in \bN\}$ that satisfies the following four conditions:\\
(i)
$\mcF^\lambda R_m\subseteq \mcF^{\lambda'}R_m$, if $\lambda\ge \lambda'$;\\
(ii)
$\mcF^\lambda R_m=\bigcap_{\lambda'<\lambda}\mcF^{\lambda'}R_m$; \\ 
(iii) 
$\mcF^\lambda R_m\cdot \mcF^{\lambda'} R_{m'}\subseteq \mcF^{\lambda+\lambda'}R_{m+m'}$, for $\lambda, \lambda'\in \bR$ and $m, m'\in \bN$; \\
(iv)
There exist $e_-, e_+\in \bZ$ such that $\mcF^{m e_-} R_m=R_m$ and $\mcF^{m e_+} R_m=0$ for $m\in \bZ_{\ge 0}$. 
\vskip 0.5mm

$\cF$ is finitely generated if its extended Rees algebra $\mathcal{R}(\cF)$ is finitely generated where
\begin{equation*}
\mathcal{R}(\cF)=\bigoplus_{\lambda\in \bR}\bigoplus_{m\in \bN} t^{-\lambda} \cF^{\lambda}R_m.
\end{equation*}
In this case $\cF$ induces a degeneration of $X$ into $\mcX_0=\mathrm{Proj}\big(\bigoplus_{m,\lambda}\cF^\lambda R_m/\cF^{>\lambda}R_m\big)$. 

{
For a general $\cF$, $\{\cF^{\lambda} R_\ell; \lambda\in \bR\}$ generates a filtration $\check{\cF}^{(\ell)}$ on $R^{(\ell)}:=\bigoplus_{m\in \bN}R_{m\ell}$, which induces a non-Archimedean psh potential $\check{\phi}^{(\ell)}\in \mcH^\NA$. Define
$\phi_{\cF}=(\limsup_{\ell\rightarrow+\infty} \check{\phi}^{(\ell)})^*$ where $(\cdot)^*$ denotes the upper semicontinuous regularization. }
\end{defn}
{
\begin{exmp}
Filtration $\cF$ is a $\bZ$-filtration if $\cF^\lambda R_m=\cF^{\lceil \lambda\rceil}R_m$. 
By \cite{BHJ17, WN12, Sze15}, there is a one-to-one correspondence between test configurations equipped with relatively ample $\bQ$-polarizations and finitely generated $\bZ$-filtrations. 
Any test configuration $(\mcX, \mcL)$ defines such a filtration:
\begin{equation}\label{eq-FTC}
\cF^\lambda R_m=\{s\in R_m; t^{-\lceil \lambda \rceil} s\in H^0(\mcX, m\mcL)\}.
\end{equation}
Conversely, if $\cF$ is a finitely generated $\bZ$-filtration, then $(\mcX:=\mathrm{Proj}_{\bC[t]}(\mathcal{R}(\check{\cF}^{(\ell)})), \frac{1}{\ell}\mathcal{O}_{\mcX}(1))$ is a test configuration for $\ell$ sufficiently divisible.  
\end{exmp}
}

\begin{exmp}
In the Definition \ref{def-TC} of test configurations, if we do not require $\mcL$ to be $\pi$-semiample, then we call $(\mcX, \mcL)$ a model (of $(X\times\bC, p_1^*L)$). The same definition in \eqref{eq-FTC} defines a filtration also denoted by $\cF_{(\mcX, \mcL)}$. However in general the filtration is not finitely generated anymore. 
Fix any model $(\mcX, \mcL)$ such that $\ocL$ is big over $\ocX$  (we call such $(\mcX, \mcL)$ a big model). In \cite{Li21} we obtained the following formula for the non-Archimedean Monge-Amp\`{e}re measure of $\phi=\phi_{(\mcX, \mcL)}:=\phi_{\cF_{(\mcX, \mcL)}}$ which generalizes \eqref{eq-TCMA}: 
\begin{equation}\label{eq-modMA}
\MA^\NA(\phi)=\sum_i b_i\big( \big\la \ocL^{\cdot n}\big\ra \cdot E_i \big) \delta_{v_i}.
\end{equation}
Here for any divisor $D$, we use the notion of positive intersection product introduced in \cite{BFJ09}:
\begin{align*}
&\la\ocL^{\cdot n+1}\ra=\vol(\ocL)=\lim_{m\rightarrow+\infty} \frac{h^0(\ocX, m \ocL)}{\frac{m^{n+1}}{(n+1)!}} \;\; \text{ and }\;\; \la \ocL^{\cdot n}\ra \cdot D=\frac{1}{n+1}\frac{d}{dt}\Big|_{t=0} \vol(\ocL+t D), 
\end{align*} 
\end{exmp}

\begin{exmp}
Any $v\in X^{\mathrm{div}}_\bQ$ defines a filtration: for any $\lambda\in \bR$ and $m\in \bZ_{\ge 0}$, define:
\begin{equation}\label{eq-filval}
\cF^\lambda_v R_m=\{s\in R_m; v(s)\ge \lambda\}. 
\end{equation}
Boucksom-Jonsson proved in \cite{BoJ18b} that $\MA^\NA(\phi_{\cF_v})=\bV\cdot \delta_v$. 
\end{exmp}

\begin{exmp}\label{exmp-twist}
Assume a torus $\tilde{\bT}\cong (\bC^*)^r$-acts on $(X, L)$. Then we have a weight decomposition $R_m=\bigoplus_{\alpha\in \bZ^r} R_{m,\alpha}$. 
For any $\xi\in \tilde{N}_\bR$, we can define the $\xi$-twist of a given filtration:
$\cF^\lambda_\xi R_m=\cF^{\lambda-\la \alpha, \xi\ra}R_{m,\alpha}$.
On the other hand, there is an induced $\tilde{N}_\bR$-action on $(X^\NA)^{\tilde{\bT}}$ which sends $(\xi, v)$ to $v_\xi\in (X^\NA)^{\tilde{\bT}}$ determined by the following condition: if $f\in \bC(X)_\alpha$ which means $f\circ\mathrm{t}^{-1}=\mathrm{t}^\alpha \cdot f$ for any $\mathrm{t}\in \tilde{\bT}$, then $v_\xi(f)=\la \alpha, \xi\ra+v(f)$. We then have the following formula $\MA^\NA(\phi_{\cF_\xi})=(-\xi)_* \MA^\NA(\phi_\cF)$ (see \cite{Li19, Li20}). 
\end{exmp}
Generalizing the case of test configurations, Boucksom-Jonsson showed that the non-Archimedean functionals from \eqref{eq-ENA}-\eqref{eq-MabNA} are well defined for all $\phi\in (\cE^1)^\NA$ by using integrals over $X^\NA$ mentioned before (for example, for $\bfH^\NA$ use \eqref{eq-HNA}).  
\begin{exmp}\label{exmp-Efil}
For any filtration $\cF$, it is known that $\phi_{\cF}\in (\cE^1)^\NA$. Following \cite{BHJ17}, define 
\begin{equation*}
\vol(\cF^{(t)})=\lim_{m\rightarrow+\infty} \frac{\dim_\bC \cF^{mt}R_m}{m^n/n!}.
\end{equation*}
Then $\bfE^\NA$ is an `expected vanishing order' with respect to $\cF$ (see \cite{BoJ18b}):
\begin{equation}\label{eq-ENAfil}
\bfE^\NA(\phi_{\cF})=\frac{1}{\bV}\int_{\bR} t (-d\vol(\cF^{(t)})).
\end{equation}
\end{exmp}
Similar to Theorem \ref{thm-Freg}, we also have important regularization properties:
\begin{theorem}[\cite{BoJ18a}]
For any $\phi\in (\cE^1)^\NA$, there exist $\{\phi_k\}_{k\in \bN}\subset \mcH^\NA$ (i.e. $\phi_k=\phi_{(\mcX_k, \mcL_k)}$ for a test configuration $(\mcX_k, \mcL_k))$ such that  $\phi_k\rightarrow \phi$ in the strong topology and 
$\bfF^\NA(\phi_k)\rightarrow \bfF^\NA(\phi)$ for $\bfF\in \{\bfE, \Lam, \bfE^{K_X}\}$. 
\end{theorem}
Boucksom-Jonsson conjectured that the same conclusion should also hold for $\bfH^\NA$. This conjecture is still open in general and it is important in the non-Archimedean approach to the YTD conjecture. We have made progress in this direction.
\begin{theorem}[\cite{Li20, Li21}]\label{thm-appmodel}
$\mathrm{(1)}$  For any $\phi\in (\cE^1)^\NA$, there exist models $\{(\mcX_k, \mcL_k)\}_{k\in \bN}$ such that $\phi_k=\phi_{(\mcX_k, \mcL_k)}\rightarrow \phi$ in the strong topology and $\bfH^\NA(\phi_k)\rightarrow \bfH^\NA(\phi)$. \\ 
$\mathrm{(2)}$  
For any big model $(\mcX, \mcL)$, we have the following formula that generalizes \eqref{eq-MabNA}:
\begin{equation*}
\bfM^\NA(\mcX, \mcL)=\frac{1}{\bV}\big\la \ocL^{\cdot n}\big\ra\cdot K_{\ocX/\bP^1}+\frac{\ud{S}}{(n+1)\bV}\big\la \ocL^{\cdot n+1}\big\ra. 
\end{equation*}
\end{theorem}
The idea for proving the first statement is similar to the Archimedean setting in \cite{BDL17}. First we regularize the measure $\MA^\NA(\phi)$ with converging entropy. In fact we find a way to regularize it by using measures supported at finitely many points in $X^{\mathrm{div}}_\bQ$. 
Then we use the solution of non-Archimedean Monge-Amp\`{e}re equations obtained in \cite{BFJ15} to get the wanted potentials which are known to be associated to models.  
However, in the non-Archimedean case, there is not yet a characterization of measures associated to test configurations which prevents us from regularizing via test configurations.  
The second statement in Theorem \ref{thm-appmodel} follows from the formula \eqref{eq-modMA}, and it prompts us to propose the following algebro-geometric conjecture which would strengthen the classical Fujita approximation theorem. 
\begin{conj}\label{conj-Fujapp}
Let $\ocX$ be a smooth $(n+1)$-dimensional smooth projective variety. Let $\ocL$ be a big line bundle over $\ocX$. Then there exist birational morphisms $\mu_k: \ocX_k\rightarrow \ocX$ and decompositions $\mu_k^*\ocL=\ocL_k+E_k$ in $N^1(\ocX)_\bQ$ with $\ocL_k$ semiample and $E_k$ effective such that
\begin{equation*}
\lim_{k\rightarrow+\infty} \ocL_k^{\cdot n+1}=\vol(\ocL), \quad \lim_{k\rightarrow+\infty} \ocL_k^{\cdot n}\cdot K_{\ocX_k}=\frac{1}{n+1}\frac{d}{dt}\vol(\ocL+t K_{\ocX})\Big|_{t=0}=:\la \ocL^{\cdot n}\ra\cdot K_{\ocX}.
\end{equation*}

\end{conj}
It is easy to show that this conjecture is true if $\ocL$ admits a birational Zariski decomposition. The author verified this conjecture for certain examples of big line bundles due to Nakamaya which do not admit such decompositions. Y. Odaka observed that when $(\mcX, \mcL)$ is a big model of a polarized spherical manifold (for example a polarized toric manifold), $\ocX$ is a Mori dream space which implies that $\ocL$ admits a Zariski decomposition and hence the above conjecture holds true.  

\subsection{Stability of Fano varieties}
In this section, we assume that $X$ is a $\bQ$-Fano variety (i.e. $-K_X$ is an ample $\bQ$-line bundle and $X$ has at worst klt singularities).
Corresponding to \eqref{eq-Dvphi}, we have a non-Archimedean $\bfD$ functional. For general test configurations it first appeared in Berman's work \cite{Berm15} and was reformulated in \cite{BHJ17} using non-Archimedean potentials:
\begin{align*}
&\bfL^\NA(\mcX, \mcL)=\inf_{v\in X^{\mathrm{div}}_\bQ}(A_X(v)+\phi_{(\mcX, \mcL)}(v)), \quad \bfD^\NA(\mcX, \mcL)=-\bfE^\NA(\mcX, \mcL)+\bfL^\NA(\mcX, \mcL). 
\end{align*} 
The notion of Ding-stability and uniform Ding-stability are defined if $\bfM^\NA$ is replaced by $\bfD^\NA$ in Definition \ref{def-Kst} and \ref{def-uniKst}.  
In general we have the inequality:
$\bfM^\NA(\mcX, \mcL)\ge \bfD^\NA(\mcX, \mcL)$.  
For Fano varieties, special test configurations play important roles. A test configuration $(\mcX, \mcL)$ is called special if the central fibre $\mcX_0$ is a $\bQ$-Fano variety and $\mcL=-K_{\mcX/\bP^1}$. 
For special test configurations, we have $\bfD^\NA=\bfM^\NA=-\bfE^\NA=:\Fut_{\mcX_0}(\xi)$, the last quantity
being the Futaki invariant on $\mcX_0$ for the holomorphic vector field $\xi$ that generates the $\bC^*$-action. 
The importance of special test configurations was first pointed out in Tian's work \cite{Tia97} motivated by compactness results from metric geometry. The following results show their importance from the point of view of algebraic geometry: 
\begin{theorem}[\cite{Fuj19, LX14, Li19}, see also \cite{BBJ18}]\label{thm-special}
For any $\bQ$-Fano variety, K-stability is equivalent to Ding-stability, and they are equivalent to 
 K-stability or Ding-stability over special test configurations.
Moreover, the same conclusion holds true if stability is replaced by semi-stability, polystability, or reduced uniform stability. 
\end{theorem}
The proofs of these results depend on a careful process of Minimal Model Program first used in \cite{LX14} to transform any given test configuration into a special one. 
Moreover crucial calculations show that the relevant invariants such as $\bfM^\NA$ or $\bfD^\NA$ decrease along the MMP process. 
Theorem \ref{thm-special} leads directly to a valuative criterion for K-stability. To state it we first define for any 
$v\in X^{\mathrm{div}}_\bQ$ an invariant (see Example \ref{exmp-Efil}):
\begin{equation}\label{eq-SLv}
S_L(v):=\frac{1}{\bV}\int_0^{+\infty} \vol(\cF^{(t)}_v)dt=\frac{1}{\bV}\int_\bR t (-d\vol(\cF^{(t)}_v))=\bfE^\NA(\phi_{\cF_v}). 
\end{equation}
Let $\tilde{\bT}$ be a maximal torus of $\Aut(X)$ and $(X^{\mathrm{div}}_\bQ)^{\tilde{\bT}}$ be the set of $\tilde{\bT}$-invariant divisorial valuations. 
Define the following invariant ($(\xi, v)\mapsto v_\xi$ is the action appeared in Example \ref{exmp-twist})
\begin{equation*}
\delta(X)=\inf_{v\in X^\textrm{div}_\bQ}\frac{A_X(v)}{S_X(v)}, \quad  \quad
\delta_{\tilde{\bT}}(X)=\inf_{v\in (X^{\textrm{div}}_\bQ)^{\tilde{\bT}}}\sup_{\xi\in \tilde{N}_\bR} \frac{A_X(v_\xi)}{S_X(v_\xi)}.
\end{equation*}
\begin{theorem}\label{thm-valcri}
\begin{enumerate}
\item (\cite{Li17, Fuj19}) $X$ is K-semistable if $\delta(X)\ge 1$. 
\item (\cite{Fuj19, Fuj19b}) $X$ is uniformly K-stable if and only if $\delta(X)>1$.
\item (\cite{BX19, Li17, Fuj19}) $X$ is K-stable if and only if $A_X(v)>S(v)$ for any non-trivial $v\in X^{\mathrm{div}}_\bQ$. 
\item (\cite{Li19}) $X$ is reduced uniformly K-stable if and only if $\delta_{\tilde{\bT}}(X)>1$. 
\end{enumerate}
\end{theorem}
To get these, it was pointed out in \cite{BHJ17} that for a special test configuration $(\mcX, \mcL)$, the valuation $\ord_{\mcX_0}$ of the function field $\bC(X\times \bC)$ restricts to become a divisorial valuation $v\in X^{\mathrm{div}}_\bQ$. A crucial observation is then made in \cite{Li17}: 
$\cF_{(\mcX, \mcL)}^\lambda R_m=\cF_{v}^{\lambda+m A(v)}R_m$ (see \eqref{eq-filval}). 
This implies $\vol(\cF^{(t)}_{(\mcX, \mcL)})=\vol(\cF_v^{(t+A(v)})$,
which together with \eqref{eq-ENAfil} leads to: \footnote{The original argument in \cite{Li17} also explicitly relates the filtration $\cF_{(\mcX, \mcL)}$ to a filtration of the section ring of $\mcX_0$ induced by the $\bC^*$-action. 
}
\begin{align}
\bfM^\NA(\mcX, \mcL)&=\bfD^\NA(\mcX, \mcL)=A(v)-\bfE^\NA(\cF_v)=A(v)-S(v). \label{eq-beta}
\end{align}
This together with Theorem \ref{thm-special} gives the sufficient condition for the K-(semi)stability. The criterion for uniform K-stability follows from similar argument and K. Fujita's inequality:  $\frac{1}{n}S(v)\le \bfJ^\NA(\cF_v)\le n S(v)$ (\cite{Fuj19b}). 
For reduced uniform stability, another identity $A_X(v_\xi)-S(v_\xi)=A_X(v)-S(v)+\Fut_X(\xi)$ proved in \cite{Li19} is needed. 

As we will see in section \ref{sec-YTD}, a main reason for introducing the (reduced) uniform K-stability is that it is much easier to use in making connection with the (reduced) coercivity in the complex analytic setting. We now have the following fundamental result:
\begin{theorem}[\cite{LXZ21}]\label{thm-LXZ}
Let $X$ be a $\bQ$-Fano variety. 
$X$ is K-stable if and only if $X$ is uniformly K-stable.
More generally, $X$ is reduced uniformly stable if and only $X$ is K-polystable. Moreover these statements hold true for any log Fano pair. 
\end{theorem}
This is achieved by several works. First, according to a work of Blum-Liu-Xu (\cite{BLX19}), divisorial valuations on $X$ associated to special test configurations are log canonical places of complements. By deep boundedness of Birkar and Haccon-McKernan-Xu, it was also shown that there exists a quasi-monomial valuation (i.e. a monomial valuation on a smooth birational model) that achieves the infimum defining $\delta(X)$ (or more generally for $\delta_{\tilde{\bT}}(X)$). Then the main problem becomes proving a finite generation property for the minimizing valuation, which is achieved by using deep techniques from birational algebraic geometry in \cite{LXZ21}. In fact in the past several years, the algebraic study of K-stability for Fano varieties has flourished and there are many important results which answer fundamental questions in this subject. We highlight two such achievements:\\
\hskip 1mm
\textbf{(1)} Algebraic construction of projective moduli space of K-polystable Fano varieties.  This is achieved in a collection of works, settling different issues in the construction including boundedness, separatedness, properness and projectivity. Moreover concrete examples of compact moduli spaces have been identified. 
 We refer to \cite{Xu20b, LXZ21} for extensive discussions on related topics.  \\
\textbf{(2)}
Fujita-Odaka \cite{FO18} introduced quantizations of the $\delta(X)$ invariant: for each $m\in \bN$, 
\begin{equation*}\label{eq-deltam}
\delta_m(X)=\inf \big\{ \lct(X, D); D \text{ is of $m$-basis type} \big\}
\end{equation*}
where $D$ is of $m$-basis type if $D=\frac{1}{m N_m}\sum_{i=1}^{N_m} \{s_i=0\}$ where$\{s_i\}$ is a basis of $H^0(X, mL)$. 
Blum-Jonsson \cite{BlJ20} proved $\lim_{m\rightarrow+\infty}\delta_m(X)=\delta(X)$. This provides a practical tool to verify uniform stability of Fano varieties. Ahmadinezhad-Zhuang \cite{AZ20} further introduced new techniques for estimating the $\delta_m$ and $\delta$ invariant which lead to many new examples of K-stable Fano varieties. 
All of these culminate in the recent determination of deformation types of smooth Fano threefolds that contain K-polystable ones (see \cite{3foldgroup}). 

In another direction, Han-Li \cite{HL20b} establishes a valuative criterion for $g$-weighted stability, corresponding to the study of $g$-solitons. A key idea in such an extension is using a fibration technique for a polynomial weight (as motivated by the theory of equivariant de Rham cohomology) and then using the Stone-Weierstrass approximation to deal with the general $g$. 
Moreover there is a notion of stability for Fano cones introduced earlier by Collins-Sz\'{e}kelyhidi associated to Ricci-flat K\"{a}hler cone metrics.
It is shown recently that this stability of Fano cones is in fact equivalent to a particular $g$-weighted stability of log Fano quotients (see \cite{AJL21, Li21b}). 

The techniques developed in the study of (weighted) K-stability of Fano varieties have also been applied to treat an optimal degeneration problem that is motivated by the Hamilton-Tian conjecture in differential geometry (see \cite{Zhu21} for background of this conjecture). This is formulated as a minimization problem for valuations in \cite{HL20} which defines (cf. \eqref{eq-beta} and \eqref{eq-SLv}), for any valuation $v\in X^\NA$, 
\begin{equation*}
\tilde{\beta}(v)=A_X(v)+\log\Big(\frac{1}{\bV}\int_0^{+\infty} e^{-t} (-d \vol(\cF^{(t)}_v))\Big). 
\end{equation*}
Very roughly speaking, the $\tilde{\beta}$ functional is an anti-derivative of certain weighted Futaki invariant. 
This functional is a variant of invariants that appeared in previous works of Tian-Zhang-Zhang-Zhu, Dervan-Sz\'{e}kelyhidi and Hisamoto (see \cite{Zhu21} for more details). 
The results from \cite{HL20b, LWX20, BLXZ21} together prove the following algebraic version of Hamilton-Tian conjecture: 
\begin{theorem}\label{thm-AHT}
For any $\bQ$-Fano variety, there exists a unique quasi-monomial valuation $v_*$ that minimizes $\tilde{\beta}$, whose associated filtration $\cF_{v_*}$ is finitely generated and induces a degeneration of $X$ to a $\bQ$-Fano variety $\mcX_0$ together with a vector field $V_\xi$. Moreover $\mcX_0$ degenerates uniquely to an $e^{-\la \cdot, \xi\ra}$-weighted polystable $\bQ$-Fano variety (cf. Example \ref{exmp-KR}). 
\end{theorem}
Combined with previous works, the uniqueness part in particular confirm a conjecture of Chen-Sun-Wang about the algebraic uniqueness of limits under normalized K\"{a}hler-Ricci flows on Fano manifolds (see \cite{Sun18}). 

\subsection{Normalized volume and local stability theory of klt singularities}

A similar minimization problem for valuations was actually studied earlier in the local setting, which motivates the formulation and the proof of Theorem \ref{thm-AHT}. 
Let $(X, x)$ be a klt singularity. Denote by $\Val_{X,x}$ the space of real valuations that have center $x$. The \textit{normalized volume} functional is introduced in \cite{Li18}: for any $v\in \Val_{X, x}$, 
\begin{equation}\label{eq-hvol}
\hvol(v):=\bigg\{
\begin{array}{ll} \;
A_X(v)^n\cdot \vol(v), & \text{\; if } A_X(v)<+\infty \\
\; +\infty, & \text{\; otherwise. }
\end{array}
\end{equation}
Here $A_X(v)$ is again the log discrepancy functional and $\vol(v)$ is defined as:
\begin{equation*}
\vol(v)=\lim_{p\rightarrow+\infty}\frac{\dim_\bC (\mcO_{X,x}/\fa_p(v))}{p^n/n!}\quad \text{where} \quad \fa_p(v)=\{f\in \mcO_{X,x}; v(f)\ge p\}.
\end{equation*}
The expression in \eqref{eq-hvol} is inspired by the work of Martelli-Sparks-Yau \cite{MSY08} on a volume minimization property of Reeb vector fields associated to Ricci-flat K\"{a}hler cone metrics. In \cite{Li18} we started to consider the minimization of $\hvol$ over $\Val_{X,x}$ and define the invariant $\hvol(X, x)=\inf_{v\in \Val_{X,x}}\hvol(v)$. We proved that the invariant $\hvol(X, x)$ is strictly positive and further conjectured the existence, uniqueness of minimizing valuations which should have finite generated associated graded rings.
For a concrete example, it was shown by the author and Y. Liu that for an isolated quotient singularity $X=\bC^n/\Gamma$, $\hvol(\bC^n/\Gamma, 0)=\frac{n^n}{|\Gamma|}$ and the exceptional divisor of the standard blowup obtains the infimum. 

This minimization problem was proposed to attack a conjecture of Donaldson-Sun, which states that the metric tangent cone at any point on a Gromov-Hausdorff limit of K\"{a}hler-Einstein manifolds depends only on the algebraic structure (see \cite{Sun18}). This conjecture has been confirmed in a series of following-up papers \cite{LX20, LX18, LWX20}. 
Algebraically we have the following results regarding this minimization problem.
\begin{theorem}
$\mathrm{(1)}$
There exists a valuation that achieves the infimum in defining $\hvol(X, x)$. Moreover this minimizing valuation is quasi-monomial and unique up to rescaling. \\
$\mathrm{(2)}$
A divisorial valuation $v_*$ is the minimizer if and only if it is the exceptional divisor of a plt blowup and also the associated log Fano pair is K-semistable.
\end{theorem}
The first statement is a combination of works by Harold Blum, Chenyang Xu and Ziquan Zhuang (\cite{Blu18, Xu20, XZ20}). 
The second statement was proved in Li-Xu \cite{LX20} (see also \cite{Blu18}) by extending the global argument from \cite{LX14} to the local case, and it shows a close relationship between the local and global theory. 
In fact, it is in proving the affine cone case of this statement when valuative criterion for K-(semi)stability was first discovered in \cite{Li17}. A similar statement is true for more general quasi-monomial minimizing valuations (\cite{LX18}). However the finite generation conjecture from \cite{Li18} is still open in general, and seems to require deeper boundedness property of Fano varieties. \footnote{Recently this conjecture has been confirmed in a preprint of Xu-Zhuang: Stable degenerations of singularities, arXiv:2205.10915.}

We also like to mention that Yuchen Liu obtained a surprising local-to-global comparison inequality by generalizing an estimate of K. Fujita:
\begin{theorem}[\cite{Liu18}]
For any closed point $x$ on a K-semistable $\bQ$-Fano variety $X$, we have:
\begin{equation}\label{eq-Liuineq}
(-K_X)^{\cdot n}\le \frac{(n+1)^{n}}{n^{n}}\hvol(X,x).
\end{equation} 
\end{theorem}
For example, if $x\in X$ is a regular point, \eqref{eq-Liuineq} recovers Fujita's beautiful inequality: $(-K_X)^{\cdot n}\le (n+1)^n$ for any K-semistable $X$ (\cite{Fuj18}). 
The inequality \eqref{eq-Liuineq} has applications in controlling singularities on the varieties that correspond to boundary points of moduli spaces. In order for this to be effective, good estimates of $\hvol(X, x)$ for klt singularities need to be developed. In particular, it is still interesting to understand better the $\hvol$ invariants and associated minimizers for 3-dimensional klt singularities. For more discussion on related topics, we refer to the survey \cite{LLX20}. 

\section{Archimedean (complex analytic) theory vs. non-Archimedean theory}

\subsection{Correspondence between Archimedean and non-Archimedean objects}
In this section, we explain results to show a general philosophy that non-Archimedean objects usually encode the information of corresponding Archimedean objects `at infinity'. 

Let $(\mcX, \mcL)$ be a test configuration and $\tilde{h}$ be a smooth psh metric on $\mcL$. Via the isomorphism $(\mcX, \mcL)\times_{\bC}\bC^* \cong X\times\bC^*$, we get a path $\tilde{\Phi}=\{\tilde{\vphi}(s)\}_{s\in \bR}$ of smooth $\omega_0$-psh potentials where $s=-\log|t|^2$. With these notation, we have the following important result:
\begin{theorem}[\cite{Tia97, Tia17, PRS08, BHJ19}]\label{thm-AvsNA}
The slope at infinity of a functional $\bfF\in \{\mathbf{E}, \Lam, \bfI, \bfJ, \bfM\}$ is given by the corresponding non-Archimedean functional:
\begin{equation*}\label{eq-slopeNA}
\bfF'^\infty(\tilde{\Phi}):=\lim_{s\rightarrow+\infty}\frac{\bfF(\tilde{\vphi}(s))}{s}=\bfF^\NA(\mcX, \mcL)=\bfF^\NA(\phi_{(\mcX, \mcL)}). 
\end{equation*}
\end{theorem}
There is a more canonical analytic object associated to a test configuration. Recall that by a geodesic ray $\Phi=\{\vphi(s)\}_{s\in [0, +\infty)}$ we mean that $\Phi|_{[s_1, s_2]}$ is a geodesic connecting $\vphi(s_1), \vphi(s_2)$ for any $s_1, s_2\in [0, \infty)$ (see \eqref{eq-geodE1}). 
\begin{theorem}[\cite{PS07}]\label{thm-PS}
For any test configuration $(\mcX, \mcL)$ for $(X, L)$, there exists a geodesic ray $\Phi_{(\mcX, \mcL)}$ emanating from any given smooth potential $\vphi_0$. 
\end{theorem}
On the other hand, recall that there is a non-Archimedean potential associated to $(\mcX, \mcL)$ (see \eqref{eq-phiTC}). Berman-Boucksom-Jonsson proved that there is a direct relation between geodesic rays and non-Archimedean potentials.  First they showed that any geodesic ray $\Phi$ defines a non-Archimedean potential (cf. \eqref{eq-phiTC}): 
\begin{equation*}
\Phi_\NA(v):=-G(v)(\Phi), \quad \text{for any } v\in X^{\mathrm{div}}_\bQ
\end{equation*}
where $G(v)(\Phi)$ is the generic Lelong number of $\Phi$ considered as a singular quasi-psh potential on a birational model where the center of the valuation $G(v)$ is a prime divisor. 
\begin{theorem}[\cite{BBJ18}]\label{thm-BBJ}
$\mathrm{(1)}$ The map $\Phi\mapsto\Phi_\NA$ has the image contained in $(\cE^1)^\NA$.
Conversely, for any $\phi\in (\cE^1)^\NA$, there exists a geodesic ray denoted by $\gamma(\phi)$ that satisfies $\gamma(\phi)_\NA=\phi$. 
 \\
$\mathrm{(2)}$
For any geodesic ray $\Phi$, $\hat{\Phi}=\gamma(\Phi_\NA)$ satisfies 
$\hat{\Phi}_\NA=\Phi_\NA\in (\cE^1)^\NA$ and $\hat{\Phi}\ge \Phi$.  \\ 
$\mathrm{(3)}$
For $\Phi=\gamma(\phi)$ with $\phi\in (\cE^1)^\NA$, $\bfE'^\infty(\Phi)=\bfE^\NA(\phi)$, and there exists a sequence of test configurations $(\mcX_m, \mcL_m)$ such that $\Phi$ is the decreasing limit of $\Phi_{(\mcX_m, \mcL_m)}$ (see Theorem \ref{thm-PS}).   
\end{theorem}
Berman-Boucksom-Jonsson proved $\Phi_\NA\in \PSH^\NA$ by blowing up multiplier ideal sheaves $\{\mcJ(m\Phi)\}_{m\in\bN}$ and using their global generation properties to construct test configurations $\{(\mcX_m, \mcL_m)\}$ such that $\phi_{(\mcX_m, \mcL_m)}$ decreases to $\Phi_\NA$. Because of the second statement, any geodesic ray $\gamma(\phi)$ with $\phi\in (\cE^1)^\NA$ is called \textit{maximal} in \cite{BBJ18}. 
By the last statement, maximal geodesic rays can be approximated by (geodesic rays associated to) test configurations. 
Moreover when $\phi=\phi_{(\mcX, \mcL)}\in \mcH^\NA$, $\gamma(\phi)$ coincides with the geodesic ray from Theorem \ref{thm-PS}. 
Further useful properties of maximal geodesic rays are known (cf. Theorem \ref{thm-AvsNA}):
\begin{theorem}[\cite{Li20}]\label{thm-slope}
Let $\Phi$ be a maximal geodesic ray. \\
$\mathrm{(1)}$
We have the identity 
$(\bfE^{-Ric(\omega_0)})'^\infty(\Phi)=(\bfE^{K_X})^\NA(\Phi_\NA)$. \\
$\mathrm{(2)}$
$\bfH'^\infty(\Phi)\ge \bfH^\NA(\Phi_\NA)$. Moreover if $\Phi=\Phi_{(\mcX, \mcL)}$ is associated to a test configuration, then $\bfH'^\infty(\Phi)=\bfH^\NA(\Phi_\NA)$. 
\end{theorem}
It is natural to conjecture that $\bfH'^\infty(\Phi)= \bfH^\NA(\Phi_\NA)$ always holds for any maximal geodesic ray $\Phi$. 
This is implied by the algebraic Conjecture \ref{conj-Fujapp} according to \cite{Li21, Li20}.

As pointed out in \cite{BBJ18}, by a construction of Darvas, there are abundant non-maximal geodesic rays. In fact analogous local examples have been used by the author to disprove a conjecture of Demailly on Monge-Amp\`{e}re mass of psh singularities. It is thus a surprising fact that maximal geodesic rays are the only ones of interest in the cscK problem. 
\begin{theorem}[\cite{Li20}]\label{thm-maximal}
If a geodesic ray $\Phi$ satisfies $\bfM'^\infty(\Phi)<+\infty$, then $\Phi$ is maximal. 
\end{theorem}
Note that $\bfM'^\infty(\Phi)=\lim_{s\rightarrow+\infty}\frac{\bfM(\vphi(s))}{s}$ exists by Theorem \ref{thm-convex}. 
This result resolves a difficulty raised in Boucksom's ICM talk \cite{Bou18}, and implies that destabilizing geodesic rays can always be approximated by  test configurations, thus giving a very strong evidence for the validity of Yau-Tian-Donaldson Conjecture \ref{conj-YTD}.
The proof of Theorem \ref{thm-maximal} starts with an equi-singular property $\int_{X\times \{|t|<1\}} e^{-\alpha(\hat{\Phi}-\Phi)}<+\infty$ for any $\alpha>0$, and then uses Jensen's inequality together with a comparison principle for the $\bfE$ functional to get a contradiction with the finite slope assumption if $\hat{\Phi}=\gamma(\Phi_\NA)\neq\Phi$. 

\subsection{Yau-Tian-Donaldson conjecture for general polarized manifolds}\label{sec-YTD}

The Yau-Tian-Donaldson (YTD) conjecture says that the existence of canonical K\"{a}hler metrics on projective manifolds should be equivalent to certain K-stability condition. For a general polarization it is believed that one needs to use a strengthened definition of K-stability such as Definition \ref{def-uniKst}.
In particular, we have the following version.
\begin{conj}[YTD conjecture]\label{conj-YTD}
A polarized manifold $(X, L)$ admits a cscK metric if and only if $(X, L)$ is reduced uniformly K-stable. 
\end{conj}
The implication from existence to stability is known, and follows from Theorem \ref{thm-cscKcri} and Theorem \ref{thm-AvsNA}. 
 The other direction is still open in general. However, based on the results discussed thus far, we can explain the proof of a weak version. 
\begin{theorem}[\cite{Li20}]\label{thm-wYTD}
If $(X, L)$ is uniformly stable over models, i.e. there exists $\gamma>0$ such that $\bfM^\NA(\mcX, \mcL)\ge \gamma \cdot \bfJ^\NA(\mcX, \mcL)$ for any model $(\mcX, \mcL)$, then it admits a cscK metric. 
\end{theorem}
\begin{proof}[Summary of proof]
\textbf{Step 1:} 
By Theorem \ref{thm-cscKcri}, we need to show that $\bfM$ is coercive. 
Assume that the coercivity fails. Then there exists a geodesic ray $\Phi=\{\vphi(s)\}_{s\in [0, \infty)}$ satisfying:
\begin{equation*}
\bfM'^\infty(\Phi)\le 0, \quad \bfJ'^\infty(\Phi)=1, \quad \sup(\vphi(s))=0.  
\end{equation*}
Such a destabilizing geodesic ray $\Phi$ was constructed in \cite{BBJ18, DR17} from a destabilizing sequence. In this construction, both the convexity of $\bfM$ from Theorem \ref{thm-convex} and a compactness result for potentials with uniform entropy bounds from \cite{BBEGZ} play crucial roles. \\
\textbf{Step 2:} By Theorem \ref{thm-maximal}, $\Phi$ is maximal. Set $\phi=\Phi_\NA$. By using Theorem \ref{thm-BBJ}.(iii) and Theorem \ref{thm-slope}.(1), we derive the identities:
\begin{equation*}
\bfE'^\infty(\Phi)=\bfE^\NA(\phi), \quad (\bfE^{-Ric(\omega_0)})'^\infty(\Phi)=(\bfE^{K_X})^\NA(\phi), \quad \bfJ'^\infty(\Phi)=\bfJ^\NA(\phi)
\end{equation*}
Moreover by Theorem \ref{thm-slope}.(2), $\bfH'^\infty(\Phi)\ge \bfH^\NA(\phi)$ so that $\bfM'^\infty(\Phi)\ge \bfM^\NA(\phi)$. \\
\textbf{Step 3:} By Theorem \ref{thm-appmodel}, there exist models $(\mcX_m, \mcL_m)$ such that $\phi_m=\phi_{(\mcX_m, \mcL_m)}$ converges to $\phi$ in the strong topology and
\begin{equation*}
\lim_{m\rightarrow+\infty} \bfM^\NA(\phi_m)=\bfM^\NA(\phi), \quad \lim_{m\rightarrow+\infty}\bfJ^\NA(\phi_m)=\bfJ^\NA(\phi).
\end{equation*}
\textbf{Step 4:} 
We can complete the proof by getting a contradiction to the stability assumption:
\begin{align*}
0\ge \bfM'^\infty(\Phi)\ge \bfM^\NA(\phi)&=\lim_{m\rightarrow+\infty} \bfM^\NA(\phi_m)\ge_{\mathrm{stability}} \lim_{m\rightarrow+\infty} \bfJ^\NA(\phi_m)=\bfJ^\NA(\phi)=1. 
\end{align*}
\end{proof}
There is a version of Theorem \ref{thm-wYTD} in \cite{Li20} when $\Aut(X, L)$ is continuous. Moreover it is
shown in \cite{Li21} that Conjecture \ref{conj-Fujapp} implies Conjecture \ref{conj-YTD}. As mentioned earlier, if $(X, L)$ is any polarized spherical manifold, conjecture \ref{conj-Fujapp} is true and hence in this case the YTD conjecture \ref{conj-YTD} is proved. Based on this fact, Delcroix (\cite{Del20}) obtained further refined existence results in this case. 

We should mention that Sean Paul (see \cite{Pau12}) has works that give a beautiful interpretation of the coercivity of $\bfM$-functional using a new notion of stability for pairs. However it is not clear how K-stability discussed here can directly imply his stability notion. 

\subsection{YTD conjecture for Fano varieties}

\subsubsection{Non-Archimedean approach}

Our proof of Theorem \ref{thm-wYTD} is in fact modeled on a non-Archimedean approach to the uniform YTD conjecture proposed by Berman-Boucksom-Jonsson in \cite{BBJ18}. They carried it out sucessfully for smooth Fano manifolds with discrete automorphism groups. The main advantage in the Fano case is that $\bfD^\NA$ satisfies a regularization property and can be used in place of $\bfM^\NA$ to complete the argument. Recently their work has been extended to the most general setting of log Fano pairs. 
\begin{theorem}[\cite{LTW20, LTW19, Li19}]\label{thm-YTDlog}
A log Fano pair $(X, D)$ admits a K\"{a}hler-Einstein metric if and only if it is reduced uniformly stable for all special test configurations. 
\end{theorem}
Note that this combined with Theorem \ref{thm-LXZ} also proves the K-polystable version of the YTD conjecture. 
Theorem \ref{thm-YTDlog} can be used to get examples of K\"{a}hler-Einstein metrics on Fano varieties with large symmetry groups (see for example \cite{IS17}). 
The proof of Theorem \ref{thm-YTDlog} is much more technical than \cite{BBJ18} because we need to overcome the difficulties caused by singularities.  
The first key idea is to use an approximation approach initiated in \cite{LTW20}. Consider the log resolution $\mu: X'\rightarrow X$ as in section \ref{sec-logFano} and re-organize \eqref{eq-KYKXD} as:
\begin{align*}
-K_{X'}-D_\epsilon=\frac{1}{1+\epsilon}(\mu^*(-K_X-D)+\epsilon H)=:L_\epsilon.
\end{align*}
where $H=\mu^*(-K_X-D)-\sum_k \theta_k E_k$ is ample by choosing appropriate $\{\theta_k\}$ and $D_\epsilon=\sum_k (-a_k+\frac{\epsilon}{1+\epsilon})\theta_k E_k$ with $0\le \epsilon\ll 1$. In \cite{LTW20} we considered the simple case when $a_k\in (-1,0]$ for all $k$. In this case for $0<\epsilon\ll 1$, $(X', D_\epsilon)$ is a smooth log Fano pair. A crucial calculation using the valuative criterion from Theorem \ref{thm-valcri} shows that (semi)stability of $(X, D)$ implies the uniform stability of $(X', D_\epsilon)$ for $\epsilon>0$. 
Moreover we can prove a version of YTD conjecture for $(X', D_\epsilon)$ and deduce that it admits a K\"{a}hler-Einstein metric.
Next we take a limit as $\epsilon\rightarrow 0$ to get a K\"{a}hler-Einstein metric on $(X, D)$ itself. 
The proof of this convergence depends on technical uniform potential and geometric estimates. 

In \cite{LTW19}, we dealt with the general case when $D_\epsilon$ is not necessarily effective. A key difficulty for the argument in \cite{BBJ18} to work on singular varieties is that it is not clear how to use multiplier ideal sheaves to approximate a destabilizing geodesic ray $\Phi$ when $X$ is singular. To circumvent this difficulty, we first need to perturb $\Phi$ to become a singular quasi-psh potential $\Phi_\epsilon$ on $(X'\times \bC, p'^*_1 L_\epsilon)$. Since $X'$ is smooth, we know how to approximate $\Phi_\epsilon$ by test configurations for $(X', L_\epsilon)$ thanks to \cite{BBJ18}. 
However due to the ineffectiveness of $D_\epsilon$, the remaining arguments depend more heavily on non-Archimedean analysis and some key observation on convergence of slopes. 
In \cite{Li20} we further derived the valuative criterion for reduced uniform stability and understood how the torus action induces an action on the space of non-Archimedean potentials in order to incorporate group actions in the argument. 
Note that the non-Archimedean approach a priori does not prove the statement in Theorem \ref{thm-YTDlog} involving special test configurations. Fortunately Theorem \ref{thm-special} fills this gap. 

By using the fibration and approximation techniques mentioned earlier, Theorem \ref{thm-YTDlog} has been extended to the case of $g$-soliton on log Fano pairs in \cite{HL20}.  As explained in \cite{AJL21,Li21b} this can be used prove the YTD conjecture for Ricci-flat K\"{a}hler cone metrics thanks to its equivalence to particular $g$-solitons (see section \ref{sec-RFKC}). This generalizes the result of Collins-Sz\'{e}kelyhidi on YTD conjecture for Fano cones with isolated singularities (\cite{CS19}). 

\subsubsection{Other approaches}\label{sec-YTDother}
For completeness, we briefly mention other approaches to the YTD conjecture on Fano manifolds. The classical way to solve the K\"{a}hler-Einstein equation is through various continuity methods. Traditionally one uses Aubin's continuity method involving twisted KE metrics. A more recent method uses KE metrics with edge cone singularities as proposed by Donaldson. Finally there is a K\"{a}hler-Ricci flow approach. Tian's early works showed that the most difficult part in proving the YTD conjecture by continuity methods is to establish the algebraicity of limit objects in the Gromov-Hausdorff topology, and he  had essentially reduced this difficulty to proving some partial $C^0$-estimates. 
The partial $C^0$-estimates were later proved in different settings, starting with Donaldson-Sun's work in the K\"{a}hler-Einstein case, which leads to the solution of the YTD conjecture for smooth Fano manifolds in \cite{CDS15, Tia15}. Moreover the partial $C^0$-estimates has applications in constructing moduli spaces of smoothable K\"{a}hler-Einstein varieties and proving quasi-projectivity of the moduli spaces of KE manifolds, and these applications preclude the algebraic approach mentioned earlier (see \cite{Wan19}). We refer to \cite{Don16, Tia15b} for surveys on related topics in this approach.  

Very recently, yet another quantization approach is carried out by Kewei Zhang based partly on an earlier work of Rubinstein-Tian-Zhang.
Zhang considered an analytic invariant of Moser-Trudinger type, namely
\begin{equation*}
\delta^A(X)=\sup\bigg\{c; \sup_{\vphi\in \mcH}\int_X e^{-c(\vphi-\bfE(\vphi))}<+\infty\bigg\}.
\end{equation*}
It is easy to show that the coercivity of $\bfD$-functional is equivalent to $\delta^A(X)>1$. The authors of
\cite{RTZ20} introduces a quantization $\delta^A_m(X)$ by using a quantization of $\bfE$ on the space of Bergman metrics, and further proves $\delta^A_m(X)=\delta_m(X)$. Using some deep results in complex geometry including Tian's work on Bergman kernels and Berndtsson's subharmonicity theorem, it is proved in \cite{Zha21} that $\lim_{m\rightarrow+\infty}\delta^A_m(X)=\delta^A(X)$.
Combining these discussion with the algebraic convergence result of Blum-Jonsson and the valuative criterion of uniform stability of Fujita discussed earlier, Zhang gets $\delta^A(X)=\delta(X)$ and completes the proof of uniform version of YTD conjecture for smooth Fano manifolds. 
It would be interesting to extend this approach to the more general case (i.e. Fano varieties with continuous automorphism groups). 

We finish by remarking that it is of interest to apply the ideas and methods from the above two approaches to study the YTD conjecture for general polarizations. For the approach involving partial $C^0$-estimates, the geometry is complicated by collapsing phenomenon in the Gromov-Hausdorff convergence with only scalar curvature bounds, which is very difficult to study with current techniques. 
For the quantization approach, there were some attempts by Mabuchi in several works. But the precise picture seems again unclear.


\textbf{Ackowledgement:}
This work is partially supported by NSF grant (DMS-181086) and an Alfred Sloan Fellowship. 
The author thanks Gang Tian, Chenyang Xu, Feng Wang, Xiaowei Wang, Jiyuan Han and Yuchen Liu for collaborations on works discussed in this survey. He also thanks many mathematicians for various discussions related to our works: Vestislav Apostolov, 
Robert Berman, S\'{e}bastien Boucksom, Tam\'{a}s Darvas, Simon Donaldson, Kento Fujita, Hans-Joachim Hein, Mattias Jonsson, L\'{a}szl\'{o} Lempert,
Mircea Musta\c{t}\u{a}, Yuji Odaka, Yanir Rubinstein,  Jian Song, Song Sun, Jun Yu, and many others.
A large part of the author's works discussed here were carried out at the Department of Mathematics, Purdue University, and he is grateful to his colleagues and friends there for creating a great environment of research.  


\end{document}